\documentclass[12pt, reqno,oneside]{amsart}
\usepackage[utf8]{inputenc}

\usepackage{amssymb, amsmath, amsthm, wasysym, mathrsfs}
\usepackage{hyperref}
\usepackage[alphabetic,lite]{amsrefs}
\usepackage{amscd}   
\usepackage{fullpage}
\usepackage{setspace}
\usepackage{mathtools}
\usepackage{enumitem}
\usepackage{caption}
\usepackage{subcaption}

\usepackage{ bbold }

\usepackage[usenames,dvipsnames]{color}
\newtheorem{lemma}{Lemma}[section]
\newtheorem{lemma*}{Lemma}
\newtheorem{theorem}[lemma]{Theorem}

\newtheorem{cor}[lemma]{Corollary}

\newtheorem{claim*}{Claim}

\newtheorem{defn}[lemma]{Definition}

\theoremstyle{definition}

\newtheorem*{lem}{Acknowledgements}
\newtheorem*{note}{Notation}
\newtheorem{rmk}[lemma]{Remark}

\newcommand{\Q}{{\mathbb Q}}

\newcommand{\Z}{{\mathbb Z}}

\numberwithin{equation}{section}
\numberwithin{table}{section}
\title{Distinguishing some genus one knots using finite quotients}
\author{Tamunonye Cheetham-West}

\begin{document}
\bibliographystyle{plain}

\maketitle
\begin{abstract}
    We give a criterion for distinguishing a prime knot $K$ in $S^3$ from every other knot in $S^3$ using the finite quotients of $\pi_1(S^3\setminus K)$. Using recent work of Baldwin-Sivek, we apply this criterion to the hyperbolic knots $5_2$, $15n_{43522}$, and the three-strand pretzel knots $P(-3,3,2n+1)$ for every integer $n$. 
\end{abstract}

\section{Introduction}
Finite quotients of the fundamental group are useful for distinguishing 3-manifolds; in particular, for a 3-manifold $M$, a finite quotient of $\pi_1(M)$ corresponds to the deck group of a finite-sheeted regular cover of $M$. If the fundamental groups of two 3-manifolds $M$ and $N$ have different finite quotients, then $\pi_1(M)\not\cong\pi_1(N)$ and the 3-manifolds $M$ and $N$ are not homeomorphic. When $M$ is a compact 3-manifold, its fundamental group $\pi_1(M)$ is residually finite \cite{HempelRF}, and so the set $C(\pi_1(M))$ of finite quotients of $\pi_1(M)$ is non-empty and infinite.  \smallbreak One of the consequences of the residual finiteness of the fundamental groups of compact 3-manifolds is that the unknot is the only knot in $S^3$ whose knot complement has a fundamental group with only finite cyclic quotients. Boileau \cite{BoileauNotes} has conjectured that every prime knot $K\subset S^3$ is completely determined by the set of finite quotients of $\pi_1(S^3\setminus K)$, that is, if for two prime knots $J$ and $K$, $\pi_1(S^3\setminus J)$ and $\pi_1(S^3\setminus K)$ have the same set of finite quotients, then $J$ and $K$ are isotopic. By work of Boileau-Friedl \cite{BF} and Bridson-Reid \cite{BR}, it is known that the figure-eight knot $4_1$ and the trefoil knot $3_1$ are completely determined by the finite quotients of their knot groups even amongst the fundamental groups of compact 3-manifolds. Furthermore, Wilkes \cite{WilkesKnots} has shown that knots in $S^3$ whose complements are graph manifolds are distinguished by the finite quotients of their knot groups. 
\smallbreak The purpose of this note is to show:

\begin{figure}
     \centering
     \begin{subfigure}[b]{0.3\textwidth}
         \centering
         \includegraphics[scale=0.50]{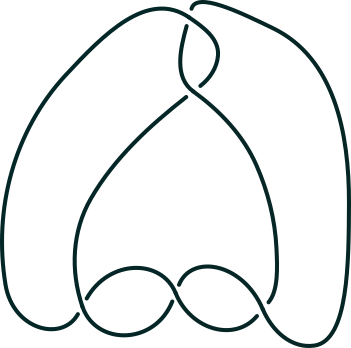}
         \caption{$5_2$}
         \label{fig:five_two}
     \end{subfigure}
     \hspace{0.5em}
     \begin{subfigure}[b]{0.3\textwidth}
         \centering
         \includegraphics[scale=0.25]{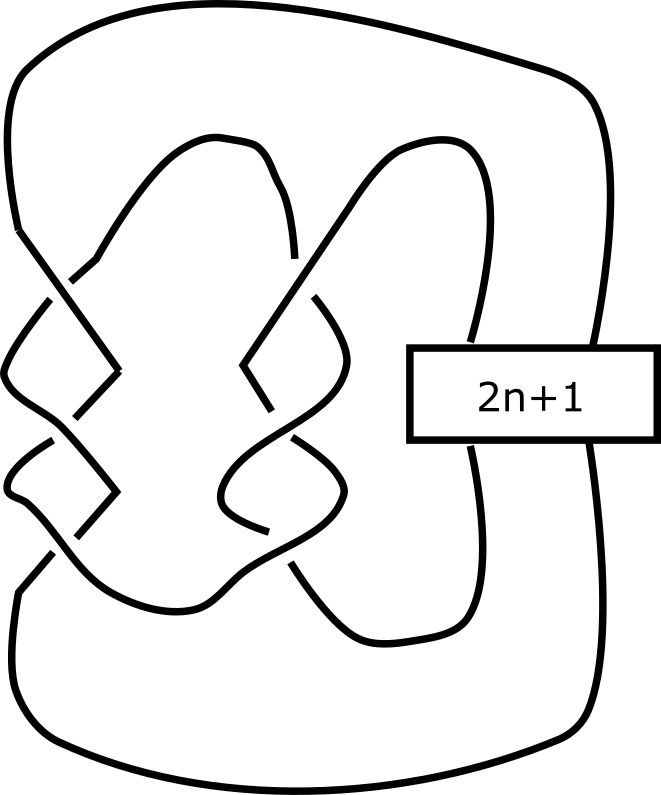}
         \caption{$P(-3,3,2n+1)$}
         \label{fig:p-332n+1}
    \end{subfigure}
     \hspace{0.5em}
     \begin{subfigure}[b]{0.3\textwidth}
         \centering
         \includegraphics[scale=0.35]{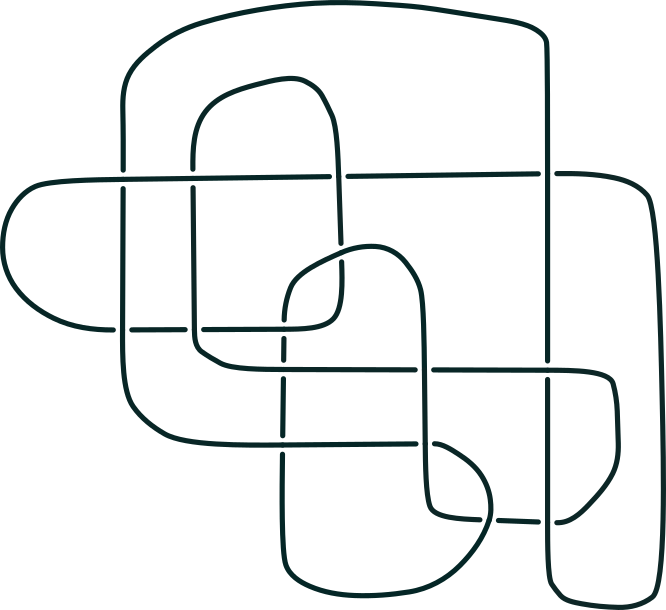}
         \caption{$15n_{43522}$}
         \label{fig:15n}
     \end{subfigure}
     \caption{Hyperbolic knots in the statement of Theorem~\ref{Bsivek}}
     \label{fig:knots}
\end{figure}

\begin{theorem}\label{Bsivek}
Let $K$ be the knot $5_2$ (shown in Figure~\ref{fig:knots}A), one of the hyperbolic pretzel knots $P(-3,3,2n+1)\,\,(n\in\Z)$ (Figure~\ref{fig:knots}B), or the knot $15n_{43522}$ (shown in Figure~\ref{fig:knots}C), then $K$ is distinguished from every other knot in $S^3$ by the finite quotients of $\pi_1(S^3\setminus K)$.
\end{theorem}
Theorem~\ref{Bsivek} will follow from our next theorem, using recent work of Baldwin-Sivek \cite{BaldwinSivek}\cite{BaldwinSivek2}, and the work of Wilkes \cite{wilkes2018} (as described in Section~\ref{proofs}). To state the theorem, we recall the definition of a {\it characterizing slope} $\alpha\in \Q$ for a knot $K$. For a knot $K\subset S^3$, let $S^3_\alpha(K)$ be the 3-manifold obtained by $\alpha-$Dehn surgery on $K$.
\begin{defn}\label{charslope}
A slope $\alpha$ is a characterizing slope for a knot $K\subset S^3$ if for any knot $J\subset S^3$, $S^3_\alpha(J)\cong S^3_\alpha(K)$ if and only if $J$ is isotopic to $K.$
\end{defn}

\begin{theorem}\label{mainthm}
Let $K$ be a hyperbolic knot in $S^3$ for which
\begin{enumerate}
    \item 0 is a characterizing slope for $K$,
    \item $S^3_0(K)$ is distinguished from every other compact, irreducible 3-manifold by the finite quotients of $\pi_1(S^3_0(K))$.
\end{enumerate}
then $K$ is distinguished from other knots in $S^3$ by the finite quotients of $\pi_1(S^3\setminus K)$. 
\end{theorem}

We point out that every knot in $S^3$ has infinitely many characterizing slopes \cite{Lackenby}. On the other hand, it is known that some knots have infinitely many non-characterizing (integral) slopes \cite{BM}. Our note relies on recent work of Baldwin-Sivek \cite{BaldwinSivek}\cite{BaldwinSivek2} showing that 0 is a characterizing slope for a family of genus 1 knots. Prior to this, the only knots previously known to have 0 as a characterizing slope were the unknot (Property R), the trefoil, and the figure-eight knot \cite{GabaiR}. 
\medbreak Theorem~\ref{mainthm} also gives a different proof that the figure-eight knot is distinguished from every other knot in $S^3$, recovering a result in \cite{BF} and \cite{BR} (see Section~\ref{proofs}). 
\begin{cor}\label{newBRBF}
The knot $4_1$ is distinguished from every other knot in $S^3$ by the finite quotients of $\pi_1(S^3\setminus 4_1)$.
\end{cor}
\begin{lem}
{\it The author thanks his advisor Alan Reid, and Ryan Spitler for many helpful conversations and for their support. The author also thanks John Baldwin and Steven Sivek for helpful conversations, for the proof of Theorem~\ref{BaldwinSivek3}, and for comments that improved this note.} 
\end{lem}
\section{Profinite Completions}
For a finitely generated, residually finite group $G$, we can organize the set of finite quotients of $G$ into an inverse system whose inverse limit $\widehat{G}$ is a finitely generated, profinite group, the {\it profinite completion} of $G$. There is a bijective correspondence between the finite index subgroups of $G$ and the open subgroups of $\widehat{G}$. There is a canonical inclusion $G\hookrightarrow \widehat{G}$ that has dense image because $G$ is residually finite. 
\medbreak Moreover, the assignment of finitely generated groups to their profinite completions is functorial, and an epimorphism of groups $G\twoheadrightarrow H$ will induce a continuous epimorphism of the profinite completions $\widehat{G}\twoheadrightarrow \widehat{H}$. The profinite completion completely captures the data of finite quotients of a group. In particular, two residually finite groups $G$ and $H$ will have isomorphic profinite completions if and only if they have the same set of finite quotients $C(G)=C(H)$ (\cite{RZ}, Corollary 3.2.8). 
\medbreak We say that a finitely generated group is {\it profinitely rigid} if it is completely determined by its profinite completion among all finitely generated, residually finite groups. We say a compact, orientable 3-manifold $M$ is {\it relatively profinitely rigid} if $\pi_1(M)$ is completely determined by its profinite completion among the fundamental groups of compact, orientable 3-manifolds. 
\begin{note}
For a subgroup $K<G$, we write the closure of $K$ in $\widehat{G}$ as $\overline{K}$. If $K$ is closed in the profinite topology on $G$, $\overline{K}\cong\widehat{K}$ and we say that $K$ is a {\it separable} subgroup of $G$. 
\end{note}
\section{Proofs}\label{proofs}
We first prove Theorem~\ref{Bsivek} and Corollary~\ref{newBRBF} assuming Theorem~\ref{mainthm}. As noted in the introduction, this relies on recent work of Baldwin-Sivek \cite{BaldwinSivek}\cite{BaldwinSivek2}. Using their classification of {\it nearly fibered} genus-1 knots, they prove:
\begin{theorem}[Theorem 1.1. \cite{BaldwinSivek}, Theorem 1.1 \cite{BaldwinSivek2}]\label{BaldwinSivek3}
Let $K$ be any of the knots $$5_2,\,\,15n_{43522},\,\,Wh^-(T_{2,3},2),\,\,Wh^+(T_{2,3},2),\,\,P(-3,3,2n+1)\,\,(n\in\Z)$$ or their mirrors. Then $0$ is a characterizing slope for $K$.
\end{theorem}
\begin{proof}
(Theorem~\ref{Bsivek}) Let $K$ be the knot $5_2$ or the knot $P(-3,3,2n+1)$ for some $n\in\Z$, as in the statement of Theorem~\ref{Bsivek}. By Theorem~\ref{BaldwinSivek3}, $0$ is a characterizing slope for $K$. Let $T$ be an incompressible torus in $S^3_0(K)$ obtained by capping off a genus 1 Seifert surface for $K$. Since $K$ is a Montesinos knot that is not the trefoil, we can apply the proof of Lemma 3.1 \cite{BaldwinSivek2} to conclude that $S^3_0(K)$ is not Seifert fibered, and in addition, when we cut $S^3_0(K)$ along $T$, we obtain the complement of the $(2,4)-$torus link $T_{2,4}\subset S^3$, which is Seifert fibered over the annulus \cite{CantwellConlon}. Thus, $S^3_0(K)$ has a non-trivial JSJ decomposition with a JSJ graph having a single vertex and cycle, which in particular is not bipartite.   
\medbreak\noindent When $K=15n_{43522}$, by Theorem~\ref{BaldwinSivek3}, $0$ is a characterizing slope for $K$. Using Regina \cite{regina}, one can check that $S^3_0(K)$ is $\{A:(2,1)\}/(\begin{smallmatrix}3 &11 \\ 2&7\end{smallmatrix})$; the result of gluing the two torus boundary components of a Seifert-fibered space with base orbifold an annulus with one cone point by the homeomorphism $(\begin{smallmatrix}3 &11 \\ 2&7\end{smallmatrix})$. Thus, $S^3_0(K)$ has a non-trivial JSJ decomposition with JSJ graph having a single vertex and cycle and therefore the JSJ graph is not bipartite. 
\medbreak\noindent Thus, by Theorem A \cite{wilkes2018}, for all $K$ in the statement of Theorem~\ref{Bsivek}, $S^3_0(K)$ is (relatively) profinitely rigid. Hence, these knots all satisfy the hypothesis of Theorem~\ref{mainthm}, and the result follows.
\end{proof}

\begin{proof}
(Corollary~\ref{newBRBF}) By \cite{GabaiR}, $0$ is a characterizing slope for the figure-eight knot $4_1$. It follows from \cite{Funar} that $S^3_0(4_1)$ is relatively profinitely rigid. Briefly, the manifold $S^3_0(4_1)$ is a torus bundle with SOLV geometry and monodromy $(\begin{smallmatrix}2 &1 \\ 1&1\end{smallmatrix})$. The eigenvalues of this monodromy are $\lambda=\frac{3\pm \sqrt{5}}{2}$, and so generate $\Q(\sqrt{5})$ which has class number 1. By Corollary 1.2 in \cite{Funar}, the number of compact 3-manifolds whose fundamental groups have the same finite quotients as $S^3_0(4_1)$ is bounded above by the class number, 1. Hence $S^3_0(4_1)$ is (relatively) profinitely rigid as claimed, and so $4_1$ satisfies the hypotheses of Theorem~\ref{mainthm} which proves the corollary.
\end{proof}

To establish Theorem~\ref{mainthm}, we will use the following lemma:
\begin{lemma}\label{mainlem}
Let $J$ and $K$ be hyperbolic knots in $S^3$ with $\lambda_J,\lambda_K$ the homological longitudes of $J$ and $K$ respectively. If $\widehat{\pi}_1(S^3\setminus J)\cong\widehat{\pi}_1(S^3\setminus K)$ then, upon identification, $\langle\lambda_J\rangle$ and $\langle\lambda_K\rangle$ have conjugate closures in $\widehat{\pi}_1(S^3\setminus K)$.  
\begin{proof}
Let $\mu_J,\mu_K$ be the meridians of $J$ and $K$ respectively, and $P_J=\langle\mu_J,\lambda_J\rangle$ and $P_K=\langle\mu_K,\lambda_K\rangle$ in $\pi_1(S^3\setminus J)$ and $\pi_1(S^3\setminus K)$ respectively. By the main theorem of \cite{Hamilton}, abelian subgroups of 3-manifold groups are separable and so $\overline{\langle\lambda_J\rangle}\cong\widehat{\langle\lambda_J\rangle}$, $\overline{P}_J\cong \widehat{P}_J<\widehat{\pi}_1(S^3\setminus J)$ and $\overline{\langle\lambda_K\rangle}\cong\widehat{\langle\lambda_K\rangle}$, $\overline{P}_K\cong\widehat{P}_K<\widehat{\pi}_1(S^3\setminus K)$. \medbreak\noindent Let $\phi:\widehat{\pi}_1(S^3\setminus J)\to \widehat{\pi}_1(S^3\setminus K)$ be an isomorphism. By Theorem 9.3 of \cite{WZ1}, it follows that $\phi(\widehat{P}_J)$ and $\widehat{P}_K$ are conjugate in $\widehat{\pi}_1(S^3\setminus K)$; so there exists $g\in\widehat{\pi}_1(S^3\setminus K)$ such that $g\phi(\widehat{P}_J)g^{-1}=\phi(\widehat{P}_K)$. Since $\langle\lambda_K\rangle$ is the intersection of $P_K$ with the kernel of the unique epimorphism $\pi_1(S^3\setminus K)\twoheadrightarrow\Z$, it follows that $\widehat{\langle\lambda_K\rangle}$ is the intersection of $\widehat{P}_K$ with the kernel of the induced epimorphism $\widehat{\pi}_1(S^3\setminus K)\twoheadrightarrow\hat{\Z}$. To see this, we observe that any element of $\widehat{P}_K$ is a $\widehat{\Z}-$linear combination of $\mu_K,\lambda_K$, and so if any element of $\widehat{P}_K$, say $\alpha=c_1\mu_K+c_2\lambda_K$ with $c_1,c_2\in\widehat{\Z}$, is in the kernel of the epimorphism $\widehat{\pi}_1(S^3\setminus K)\twoheadrightarrow\hat{\Z}$ induced by the unique epimorphism $\pi_1(S^3\setminus K)\twoheadrightarrow\Z$, $c_1=0$ because the image of $\mu_K$ in $\widehat{\Z}$ is non-trivial while the image of $\lambda_K$ is trivial. As $c_1=0$, $\alpha\in\widehat{\langle\lambda_K\rangle}$. \medbreak\noindent Next, observe that since the epimorphism $\pi_1(S^3\setminus J)\twoheadrightarrow \Z$ is unique, it can also be described as the homomorphism obtained by the following composition of maps  $$\pi_1(S^3\setminus J)\hookrightarrow \widehat{\pi}_1(S^3\setminus J)\overset{\phi}{\to}  \widehat{\pi}_1(S^3\setminus K)\twoheadrightarrow\hat{\Z}\,\,\,\,(\clubsuit)$$
where the first map is the canonical inclusion of $\pi_1(S^3\setminus J)$ into its profinite completion, and the last map $\widehat{\pi}_1(S^3\setminus K)\twoheadrightarrow\hat{\Z}$ is the epimorphism induced by the unique epimorphism $\pi_1(S^3\setminus K)\twoheadrightarrow\Z$. Furthermore, $\widehat{\langle\lambda_J}\rangle$ is the intersection of $\overline{P_J}\cong\widehat{P}_J$ with the kernel of the map $$\widehat{\pi}_1(S^3\setminus J)\overset{\phi}{\to}  \widehat{\pi}_1(S^3\setminus K)\twoheadrightarrow\hat{\Z}$$ because $\langle\lambda_J\rangle$ is the intersection of $P_J$ with the kernel of the unique epimorphism $\pi_1(S^3\setminus J)\twoheadrightarrow\Z$ which we have noted is the same map as the composition of maps $(\clubsuit)$. Since $g^{-1}\lambda_Kg$ is in the intersection of $\phi(\widehat{P}_J)$ with the kernel of the epimorphism $\widehat{\pi}_1(S^3\setminus K)\twoheadrightarrow\hat{\Z}$, $g^{-1}\lambda_Kg\in\widehat{\langle\phi(\lambda_J)\rangle}$. By interchanging $J$ and $K$, we can apply the foregoing argument to show that $g\phi(\lambda_J)g^{-1}\in\widehat{\langle\lambda_K\rangle}$ and so the closures of $\langle\phi(\lambda_J)\rangle$ and $\langle\lambda_K\rangle$ are conjugate in $\widehat{\pi}_1(S^3\setminus K)$.
\end{proof}
\end{lemma}
With this lemma in hand, we can prove Theorem~\ref{mainthm}. We recall the hypotheses; for a knot $K$, the slope 0 is characterizing for $K$ and the 0-Dehn surgery on $K$ is (relatively) profinitely rigid.   
\begin{proof}\label{proofmainthm}
(Theorem~\ref{mainthm}) Let $J$ be a knot in $S^3$ with $\widehat{\pi}_1(S^3\setminus J)\cong\widehat{\pi}_1(S^3\setminus K)$. By Theorem A of \cite{WZ1}, $J$ is a hyperbolic knot. Assuming that $Q$ is a finite quotient of $\pi_1(S^3_0(K))$, we can precompose with the Dehn-filling epimorphism $\pi_1(S^3\setminus K)\twoheadrightarrow \pi_1(S^3_0(K))$ to obtain an epimorphism from $\pi_1(S^3\setminus K)$ to $Q$ under which $\lambda_K$ maps trivially. As in the proof of Lemma~\ref{mainlem}, we can choose some identification $\phi:\widehat{\pi}_1(S^3\setminus J)\cong\widehat{\pi}_1(S^3\setminus K)$, and thereby obtain an epimorphism $\pi_1(S^3\setminus J)\twoheadrightarrow Q$ by the following composition: $$\pi_1(S^3\setminus J)\hookrightarrow \widehat{\pi}_1(S^3\setminus J)\overset{\phi}{\to}  \widehat{\pi}_1(S^3\setminus K)\twoheadrightarrow Q$$  By Lemma~\ref{mainlem}, since $\langle\phi(\lambda_J)\rangle$ and $\langle\lambda_K\rangle$ have conjugate closures and $\lambda_K$ maps trivially, $\lambda_J$ will also map trivially. It follows that $Q$ is a finite quotient of $\pi_1(S^3_0(J))$.\medbreak\noindent Applying this argument with the roles of $J$ and $K$ reversed shows that every finite quotient of $\pi_1(S^3_0(J))$ is also a finite quotient of $\pi_1(S^3_0(K))$, and we conclude that $\widehat{\pi}_1(S^3_0(J))\cong\widehat{\pi}_1(S^3_0(K))$. 
\medbreak\noindent However by the second hypothesis, $S^3_0(K)$ is (relatively) profinitely rigid and so $S^3_0(J)$ is homeomorphic to $S^3_0(K)$. By the first hypothesis, $0$ is a characterizing slope for $K$ and so $J$ and $K$ are isotopic. 
\end{proof}
\begin{rmk}
Note that the knots $Wh^+(T_{2,3},2)$ and $Wh^-(T_{2,3},2)$ are satellite knots and thus the proof of Theorem~\ref{mainthm} does not apply.
\end{rmk}
\bibliography{main}
\bigbreak\noindent Department of Mathematics,
\smallbreak\noindent Rice University,
\smallbreak\noindent Houston, TX, 77005,
\smallbreak\noindent Email: tcw@rice.edu
\end{document}